\documentclass[12pt]{amsart}
\usepackage{geometry}                
\usepackage{graphicx}
\usepackage{amssymb}
\usepackage{float,subfigure}

\theoremstyle{plain} 
\newtheorem{theorem}{Theorem}[section]
\newtheorem{lemma}[theorem]{Lemma}
\newtheorem{corollary}[theorem]{Corollary}

\theoremstyle{definition} 
\newtheorem{definition}[theorem]{Definition}
\newtheorem{remark}[theorem]{Remark}
\newtheorem{example}[theorem]{Example}

\newcommand{\Rcal}{\mathcal{R}}

\newcommand{\Z}{\mathbb{Z}}
\newcommand{\R}{\mathbb{R}}
\newcommand{\C}{\mathbb{C}}
\newcommand{\N}{\mathbb{N}}

\newcommand{\re}{\operatorname{Re}}

\def\skp#1#2{\langle {#1}, {#2} \rangle}
\def\Dds {\frac{D}{d s}}

\begin{document}
\begin{titlepage}

\begin{title}
{On Surfaces that are Intrinsically Surfaces of Revolution}
\end{title}

\author
{Daniel Freese}
\address{Daniel Freese\\Department of Mathematics\\
Liberty University\\
Lynchburg, VA 24515
\\USA}
\author{
Matthias Weber
}
\address{Matthias Weber\\Department of Mathematics\\Indiana University\\
Bloomington, IN 47405
\\USA}
\thanks{This work was partially supported by a grant from the Simons Foundation (246039 to Matthias Weber)
and by a grant from the NSF (1461061 to Daniel Freese)}

\subjclass[2010]{Primary 53C43; Secondary 53C45}
\date{\today}
\maketitle

\begin{abstract}
We consider surfaces in Euclidean space parametrized on an annular domain such that the first fundamental form and the principal curvatures are  rotationally invariant, and the principal curvature directions only depend on the angle of rotation (but not the radius).
Such surfaces generalize the Enneper surface.  We show that they are necessarily of constant mean curvature, and that the rotational speed of the principal curvature directions is constant. We classify the minimal case. The (non-zero) constant mean curvature case has been classified by Smyth.
\end{abstract}

\end{titlepage}
\section{Introduction}

The Enneper surface (see Figure \ref{fig:enneper}) is given in conformal polar coordinates as

\[
f(u,v)=\frac{1}{6} e^u 
\begin{pmatrix}
3 \cos (v)-e^{2 u} \cos (3 v) \\
 -3 \sin (v)-e^{2 u} \sin (3 v) \\
 3 e^u \cos (2 v) \ .
 \end{pmatrix}
 \]

\begin{figure}[H] 
   \centering
   \includegraphics[width=3in]{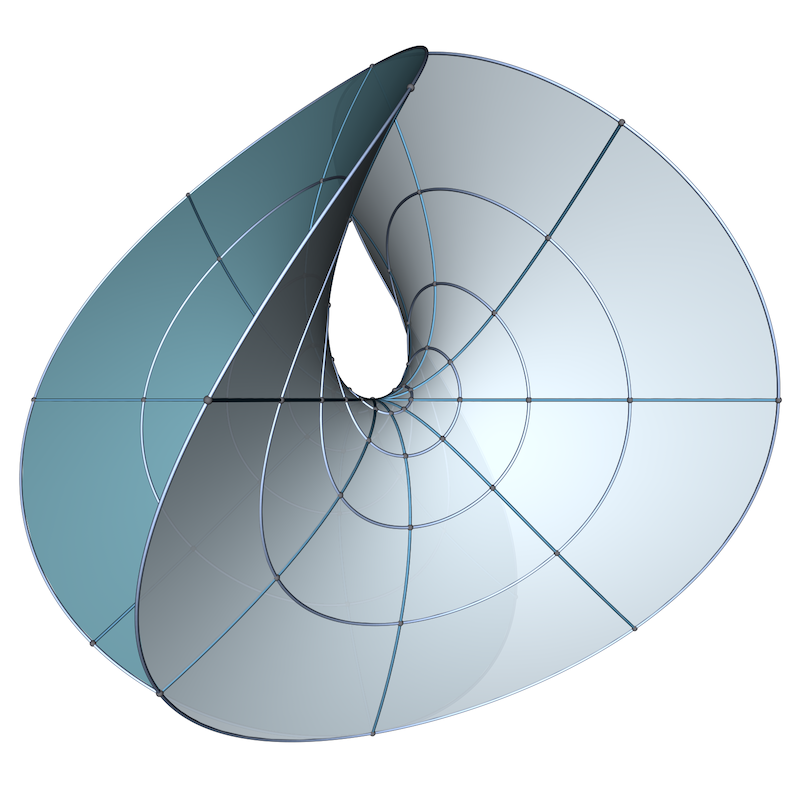} 
   \caption{The Enneper Surface}
   \label{fig:enneper}
\end{figure}

It was discovered in 1871 by Alfred Enneper \cite{enn2}. Its first fundamental form is given by

\[
I = \frac{1}{4} e^{2 u}
   \left(1+e^{2 u}\right)^2
   \begin{pmatrix} 1 &
   0 \\
 0 & 1
\end{pmatrix} \ .
\]
This means that the Enneper surface  is \emph{intrinsically} a surface of revolution (but obviously not extrinsically).

\begin{definition}\label{def:rho}
An intrinsic surface of revolution is a surface with first fundamental form of the shape
\[
I = I_\rho = \rho(u)^2
\begin{pmatrix} 1 &
   0 \\
 0 & 1
\end{pmatrix}\ ,
\]
where $\rho(u)$ is a positive function.
\end{definition}

Of course any surface of revolution is also intrinsically a surface of revolution.

The shape operator of the Enneper surface is also rather special:

\[
S = 
\frac4{\left(1+e^{2 u}\right)^2}
\begin{pmatrix}  
\cos (2 v) & -\sin (2 v) \\
 -\sin (2 v) & -\cos (2 v)
 \end{pmatrix}
=
R^{-v}
\begin{pmatrix}
 \frac{4}{\left(1+e^{2 u}\right)^2} & 0 \\
 0 & -\frac{4}{\left(1+e^{2u}\right)^2}
\end{pmatrix}
R^{v}
\]

where 
\[
R^v = 
\left(
\begin{array}{ll}
 \cos (v) & -\sin (v) \\
 \sin (v) & \cos (v)
\end{array}
\right)
\]
is the counterclockwise rotation by $v$. This is in contrast to the shape operator of a surface of revolution which always
takes diagonal form in polar coordinates. It is, however, rather special, because the principal curvature directions rotate with constant speed independent of $u$ and the principal curvatures are independent of $v$. 

\begin{figure}[H] 
   \centering
   \includegraphics[width=3in]{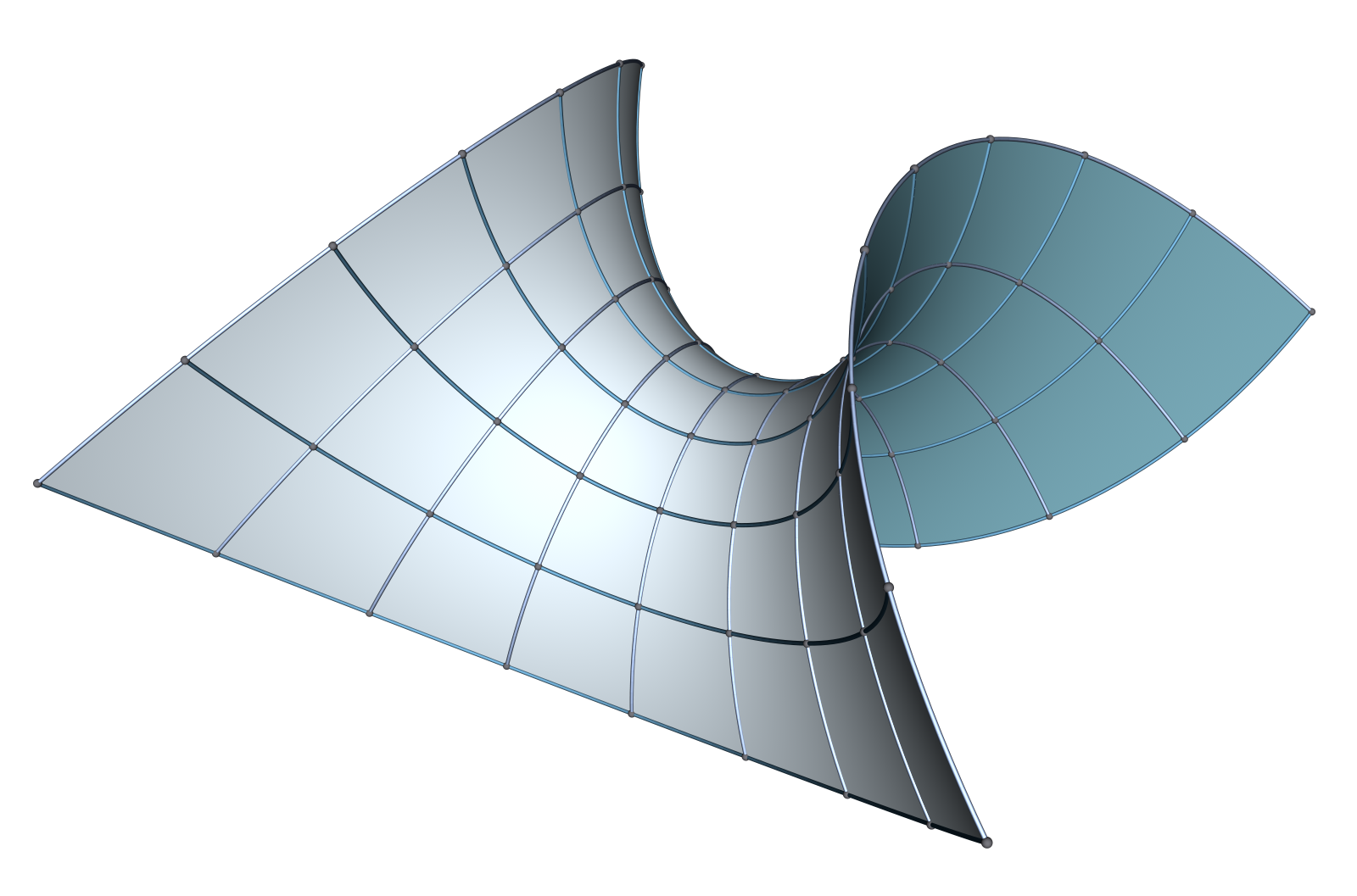} 
   \caption{The Enneper Surface with curvature lines}
   \label{fig:ennepercurvature}
\end{figure}

We are generalizing this property of the Enneper surface by introducing the following concept:

\begin{definition}\label{def:shape}
Let $\alpha:\R \to \R$ be a $C^1$-function. 
We say a surface has twist $\alpha$ if its shape operator is of the form
\begin{equation}\label{eqn:shapedef}
S = R^{-\alpha(v)}
\begin{pmatrix}
\lambda_1(u) & 0 \\
 0 & \lambda_2(u) 
\end{pmatrix}
R^{\alpha(v)}
 \ .
\end{equation}
\end{definition}

Note that this precisely means that the principal curvature directions are independent of  $u$, and the principal curvatures are independent of $v$. 

In summary, the Enneper surface is an example of an intrinsic surface of revolution with twist $\alpha(v)\equiv v$.
A standard surface of revolution, on the other hand, has twist $\alpha(v)\equiv 0$.

Now we can formulate our main theorem, which is a consequence of the Codazzi equations. We note that we generally assume our surfaces to be three times continuously differentiable.

\begin{theorem}\label{thm:cmc}
Let $\Sigma$ be  an intrinsic surface of revolution with twist function $\alpha$. 
Assume that $\alpha$ is not identically equal to 0 or any other integral multiple of $\pi/2,$ on any open interval. Assume furthermore that the surface has no open set of umbilic points. Then $\Sigma$ has constant mean curvature, and the twist function is linear $\alpha(v)=a v$.
\end{theorem}

Constant mean curvature surfaces that are intrinsic surfaces of revolution have been studied by Smyth, see \cite{smy1}.
Thus our result complements Smyth's result by replacing his assumption about constant mean curvature with a geometric assumption.

In order to begin a complete classification, we will invoke Bonnet's theorem \cite{doc1} to prove:

\begin{theorem}\label{thm:ode}
Let $\Sigma$ be  an intrinsic surface of revolution of constant mean curvature $H=\lambda_1+\lambda_2$,  first fundamental form $I_\rho$ for $\rho:(u_1,u_2)\to\R^{>0}$, and  
linear twist $\alpha(v)=a v$. Then $\rho$ satisfies the differential equation
\begin{equation}\label{eqn:rhoeqn1}
\rho '(u)^2-\rho (u) \rho ''(u)=\frac{1}{4} H^2 \rho (u)^4-b^2
   e^{4 a u} 
\end{equation}
for a constant $b$.

Vice versa, given $H$, and $\alpha(v)=a v$, a constant $b$ and $\rho$ satisfying Equation (\ref{eqn:rhoeqn1}),
define
\[
\lambda_{1,2}(u)= \frac12 H \pm b \frac{e^{2a u}}{\rho(u)^2} \ .
\]
Then the first fundamental form $I_\rho$ and the shape operator $S$ given by Equation (\ref{eqn:shapedef})  satisfy the Gauss- and Codazzi equations and thus define an intrinsic surface of revolution with constant twist $\alpha(v)=av$ and constant mean curvature $H$.
\end{theorem}

In the special case of minimal surfaces, we can achieve a complete classification.

\begin{theorem}\label{thm:min}
Let $\Sigma$ be  an intrinsic surface of revolution that is also minimal with constant twist $\alpha(v)=av$ with $a>0$.   Then $\Sigma$ belongs to an explicit 2-parameter family of minimal surfaces with Weierstrass data given by
\[
G(w) =-\frac{1}{A} w^{B}\qquad\text{and}\qquad dh  = \frac{1}{2B} w^{2 a-1} \, dw \ ,
\] 
with parameters $A$ and $B$.
\end{theorem}

In the case that the twist function is $\alpha(v)\equiv a=0$, we prove:

\begin{theorem}\label{thm:unt}
Given a conformal factor $\rho(u)$ on an interval $u_1<u<u_2$ and a constant $c$ such that $c^2 \rho (u)>|\rho '(u)|$ for all $u_1<u<u_2$. Then there is an intrinsic surface of revolution defined on the domain $(u_1,u_2)\times \R$ with first fundamental form $I_\rho$ and twist $\alpha(v)\equiv 0$. Moreover, this surface can be realized as an actual surface of revolution in $\R^3$ of the form
\[
f(u,v) = \left( g(u) \cos(c v), g(u)\sin( c v), h(u) \right)
\]
with suitable functions $g, h: (u_1,u_2)\times \R$.
\end{theorem}

The paper is organized as follows:

\begin{itemize}
\item
In section \ref{sec:gc}, we compute the Gauss- and Codazzi equations for intrinsic surfaces of revolution, reduce them to a single ODE for $\rho$, and prove Theorems \ref{thm:cmc} and \ref{thm:ode}.
\item
In section \ref{sec:min}, we specialize this equation to the minimal case, integrate the surface equations, find the Weierstrass representation of the surfaces, prove Theorem \ref{thm:min}, and give examples.
\item
In section \ref{sec:cmc}, we briefly discuss the constant mean curvature case by connecting our approach to Smyth's. While we are not able to find explicit solutions for the Smyth surfaces, we can find numerical solutions and make images.
\item
In section \ref{sec:unt}, we consider the case of twist 0, prove Theorem \ref{thm:unt}, and show that sectors of the Enneper surface are isometric to sectors (with different angle) of surfaces of revolution.
\end{itemize}

\section{Gauss- and Codazzi equations for intrinsic surfaces of revolution with twist $n>0$}\label{sec:gc}

We will apply Bonnet's theorem to determine when the first fundamental form and shape operator of an intrinsic surface of revolution with twist function $\alpha$ are induced by an actual surface in $\R^3$.

In order  to derive the Gauss- and Codazzi equations we first determine the relevant covariant derivatives. Much of this preparation is standard.

Introduce
\begin{equation}\label{eqn:UV}
U = \frac1{\rho(u)} \frac{\partial}{\partial u} \qquad V = \frac1{\rho(u)} \frac{\partial}{\partial v}
\end{equation}
as the normalized coordinate vector fields. Then we have

\begin{lemma}\label{lemma:levi}
The Levi-Civita connection of the first fundamental form $I_\rho$ is given by
\begin{alignat*}{3} 
   D_U U {}&= 0 &\qquad D_U V {}&= 0\\
   D_V U {}&= \frac{\rho'(u)}{\rho(u)^2} V &\qquad D_V V {}&= -\frac{\rho'(u)}{\rho(u)^2} U\ .\\
\end{alignat*}
\end{lemma}

\begin{proof}

By the $v$-invariance of the first fundamental form, the curves $s\mapsto (s,v)$ are geodesics, and $U$ has length $1$ with respect to the first fundamental form. This implies $D_UU=0$. Next $V= R^{\pi/2} U$ and intrinsic rotations are parallel, so that $D_U V=0$ as well.

Using that $D$ is torsion-free and metric, we compute
\begin{align*}
D_V U 	{}&= D_U V +[V,U]\\
		{}&=  d_V U - d_U V \\
		{}&= \frac{\partial}{\rho(u) \partial v}\frac{\partial}{\rho(u) \partial u} - \frac{\partial}{\rho(u) \partial u}\frac{\partial}{\rho(u) \partial v}  \\
		{}&= -\frac1{\rho(u)}  \frac{\partial}{ \partial u}\frac{\partial}{\rho(u) \partial v}  \\
		{}&= \frac1{\rho(u)}  \frac{\rho'(u)}{\rho(u)^2}\frac{\partial}{ \partial v}  \\
		{}&=   \frac{\rho'(u)}{\rho(u)^2}V  
\end{align*}

and

\begin{align*}
D_V V 	{}&= I(D_V V, U) U + I(D_V V, V) V \\
		{}&=  d_V  I( V,U) U -I( V, D_V U) U \\
		{}&=   - \frac{\rho'(u)}{\rho(u)^2} U \ . 
\end{align*}

\end{proof}

\begin{lemma}\label{lem:gauss}
The Gauss equation is equivalent to
\[
\lambda_1(u)\lambda_2(u) = \frac{\rho'(u)^2-\rho(u)\rho''(u)}{\rho(u)^4} \ .
\]
\end{lemma}

\begin{proof}
The Gauss equation gives us:
\[
I(\Rcal(U, V)V, U) = det(S) = \lambda_1(u)\lambda_2(u)
\]
Let us calculate $\Rcal(U, V)V$.
\begin{align*}
\Rcal(U, V)V &=D_U D_V V - D_V D_U V- D_{[U,V]} V\\
&=D_U\left(- \frac{\rho'(u)}{\rho(u)^2} U\right)- D_V (0)- D_{-\frac{\rho'(u)}{\rho(u)^2}V} V\\
&=-d_U\left(\frac{\rho'(u)}{\rho(u)^2}\right)U -D_U U+ \frac{\rho'(u)}{\rho(u)^2}D_V V\\
&=-\frac1{\rho(u)}\frac{\partial}{ \partial u}\left(\frac{\rho'(u)}{\rho(u)^2}\right) U + \frac{\rho'(u)}{\rho(u)^2}\left(- \frac{\rho'(u)}{\rho(u)^2} U\right)\\
&=-\frac1{\rho(u)}\left(\frac{\rho''(u)\rho(u)^2-\rho'(u)(2\rho(u)\rho'(u))}{\rho(u)^4}\right) U - \frac{\rho'(u)^2}{\rho(u)^4} U\\
&=\frac{-\rho''(u)\rho(u)+2\rho'(u)^2-\rho'(u)^2}{\rho(u)^4} U \\
&=\frac{\rho'(u)^2-\rho(u)\rho''(u)}{\rho(u)^4} U
\end{align*}
The claim follows.
\end{proof}

In order to use the Codazzi equations we continue to compute the relevant covariant derivatives.

\begin{lemma}The covariant derivatives of the twist rotation are given by
\[
D_U R^{\alpha(v)} = 0  \qquad\text{and}\qquad D_V R^{\alpha(v)} = \frac{\alpha'(v)}{\rho(u)} R^{\alpha(v)+\frac\pi2} \ .
\]
\end{lemma}
\begin{proof}
The first equations follows because intrinsic rotations by a constant angle are parallel and $\alpha$ is independent of $u$.
For the second, we use the chain rule and observe that $V=\frac 1\rho \frac{\partial}{\partial v}$. 
\end{proof}

\begin{lemma}The covariant derivatives of the eigenvalue endomorphism $\Lambda$  are given by
\begin{align*} 
  D_U \Lambda(u)  {}&=  \frac1{\rho(u)}\Lambda'(u) \\
  (D_V \Lambda(u)) U {}&=  \frac{\rho'(u)}{\rho^2(u)}(\lambda_1(u) -\lambda_2(u)) V\\
  (D_V \Lambda(u)) V  {}&=  \frac{\rho'(u)}{\rho^2(u)}(\lambda_1(u) -\lambda_2(u)) U\ .
  \end{align*}
\end{lemma}

\begin{proof}
The first equation is immediate because the frame $(U,V)$ is parallel in the $U$-direction. For the second, we compute
\begin{align*} 
 ( D_V \Lambda) U  {}&=  D_V (\lambda_1 U) - \Lambda D_V U\\
  {}&=  \lambda_1 D_V U - \lambda_2 D_V U\\
  {}&=  \frac{\rho'(u)}{\rho^2(u)}(\lambda_1(u) -\lambda_2(u)) V\ .
  \end{align*}
The third equation is proven the same way.
\end{proof}

\begin{lemma}The covariant derivatives of the shape operator are given by
\begin{align*} 
  (D_U  S) V  {}&=  \frac1{\rho(u)}R^{-\alpha(v)} \Lambda'(u) R^{\alpha(v)} V\\
  (D_V S) U {}&=  (\lambda_1(u)-\lambda_2(u))\frac{\rho'(u)-\rho(u)\alpha'(v)}{\rho^2(u)}R^{-2\alpha(v)} V \ .
  \end{align*}
\end{lemma}
\begin{proof}
In the statement, we have indicated the dependence of each functions by their respective variables. To improve legibility, we will drop the variables in the 
computations below. Observe, however, that derivatives like $\alpha'$ and $\rho'$ are always taken with respect to the proper variables.

The first equation is again immediate. For the second, we need to work harder. We begin by differentiating the definition of $S$ according to the product rule, and then use the lemmas above:
\begin{align*} 
   (D_V S) U {}&= D_V (R^{-\alpha} \Lambda R^{\alpha}) (U) \\
   {}&= \left( (D_V R^{-\alpha} )\Lambda R^{\alpha} + R^{-\alpha} (D_V \Lambda) R^{\alpha} + R^{-\alpha} \Lambda (D_V R^{\alpha})\right)(U) \\
    {}&= R^{-\alpha} \left(- \frac{\alpha'}{\rho} R^{\frac\pi2}\Lambda R^{\alpha} + (D_V \Lambda) R^{\alpha} + \frac{\alpha'}{\rho}\Lambda  R^{\alpha+\frac\pi2}\right)(U) \\
   {}&= R^{-\alpha} \Big( -\frac{\alpha'}{\rho} R^{\frac\pi2}\Lambda (\cos(\alpha)U +\sin(\alpha)V)  + \\
    {}&\qquad (D_V \Lambda) (\cos(\alpha)U +\sin(\alpha)V) + \frac{\alpha'}{\rho}\Lambda  (-\sin(\alpha)U +\cos(\alpha)V)\Big) \\  
     {}&= R^{-\alpha} \Bigg(- \frac{\alpha'}{\rho} R^{\frac\pi2} (\lambda_1\cos(\alpha)U +\lambda_2\sin(\alpha)V)  + \\
    {}&\qquad   \frac{\rho'}{\rho^2}(\lambda_1 -\lambda_2)(\cos(\alpha)V +\sin(\alpha)U) + \frac{\alpha'}{\rho} (-\lambda_1\sin(\alpha)U +\lambda_2\cos(\alpha)V)\Bigg) \\  
    {}&= (\lambda_2-\lambda_1) \left( -\frac{\alpha'}{\rho} +\frac{\rho'}{\rho^2}\right) R^{-\alpha}(-\sin(\alpha)U-\cos(\alpha)V) \\ 
    {}&= (\lambda_1-\lambda_2) \left( -\frac{\alpha'}{\rho} +\frac{\rho'}{\rho^2}\right) R^{-2\alpha}  V 
   \end{align*}
\end{proof}

\begin{corollary}
The Codazzi equations are equivalent to
\begin{align*}
( \lambda'_1+\lambda'_2)\sin(\alpha) \cos(\alpha) {}&=  0  \\
\frac1\rho (\lambda_1-\lambda_2)(\rho'-\rho\alpha') {}&=  -\lambda_1'\sin^2(\alpha)+\lambda_2' \cos^2(\alpha) \ .
\end{align*}
\end{corollary}
\begin{proof}

The Codazzi equations state that $(D_X S) Y = (D_Y S) X$ for any pair of tangent vectors $X$ and $Y$. As we are in dimension 2 and the equation is symmetric, it suffices to verify this for $X=U$ and $Y=V$. By the previous theorem, this is equivalent to
\[
{\rho} \Lambda' R^{\alpha} V = (\lambda_1-\lambda_2)(\rho'-\rho\alpha')R^{-\alpha} V  \ .
\]
Pairing both sides with $I(\cdot, R^{-\alpha} U)$ and $I(\cdot, R^{-\alpha} V)$ respectively gives
\begin{align*}
{\rho} I( \Lambda' R^{\alpha} V, R^{-\alpha} U) {}&= 0 \\
{\rho} I(\Lambda' R^{\alpha} V, R^{-\alpha} V) {}&= (\lambda_1-\lambda_2)(\rho'-\rho\alpha') \ .
\end{align*}
The first equation simplifies to
\begin{align*}
0 {}&= I( -\Lambda' (\sin(\alpha) U + \cos(\alpha) V), \cos(\alpha) U - \sin(\alpha) V) \\
&= I\left( -\lambda'_1\sin(\alpha) U + \lambda'_2 \cos(\alpha) V), \cos(\alpha) U - \sin(\alpha) V\right) \\
&= ( \lambda'_1+\lambda'_2)\sin(\alpha) \cos(\alpha)  \\
\end{align*}
and the second to
\begin{align*}
 \frac1\rho (\lambda_1-\lambda_2)(\rho'-\rho\alpha') {}&= I(\Lambda' R^{\alpha} V, R^{-\alpha} V) \\
&= I( -\lambda_1'\sin(\alpha) U + \lambda_2' \cos(\alpha) V), \sin(\alpha) U + \cos(\alpha) V) \\
&=  -\lambda_1'\sin^2(\alpha)+\lambda_2' \cos^2(\alpha)  
\end{align*}
as claimed.
\end{proof}

We are now ready to prove Theorem \ref{thm:cmc}

\begin{proof}
By assumption, the twist function $\alpha$ is not identically equal to an integral multiple of $\pi/2$ on any open interval. By the first Codazzi equation,  the mean curvature $H(u)=\lambda_1(u)+\lambda_2(u)$ is constant except possibly at isolated points. As we assume that $H$ is at least $C^1$, this implies that $H$ is constant.

This simplifies the second Codazzi equation to 
\[
\frac1{\rho(u)} (2\lambda_1(u)-H)(\rho'(u)-\rho(u)\alpha'(v)) =  -\lambda_1'(u) \ .
\]

As the right hand side is independent of $v$, so is the left hand side. This can only be the case if $\alpha'(v)$ is a constant as claimed, or that $H=2\lambda_1(u)$ on an open interval. In the latter case  we have on the same interval that $\lambda_1(u) = \lambda_2(u) = \lambda$ for a constant $\lambda$. But this means that this portion of the surface is umbilic, which we have excluded.
\end{proof}

Observe that we have not used the Gauss equations in the proof above. We will now use the second Codazzi equation to eliminate $\lambda_1$ and $\lambda_2$ from the Gauss equations.

\begin{lemma}\label{lem:eigen}
For $\alpha(v) = a v$ and $H=\lambda_1(u)+\lambda_2(u)$ a constant, the second Codazzi equation has the general solution
\[
\lambda_1(u) = \frac12 H +b \frac{e^{2a u}}{\rho(u)^2} \ ,
\]
where $b$ is any real number.
\end{lemma}

\begin{proof}
Define 
\[
\mu(u)= \rho^2(u)\left(\lambda_1(u)-\frac12H\right) \ .
\]
The second Codazzi equation is then equivalent to 

\[
\mu'(u) = 2 a \mu(u) \ .
\]
Integrating and substituting back gives the claim.
\end{proof}

The following corollary proves Theorem \ref{thm:ode}.

\begin{corollary}
A first fundamental form $I_\rho$ with $\rho=\rho(u)$ and shape operator $S$ as in Equation (\ref{eqn:shapedef}) such that $H=\lambda_1(u)+\lambda_2(u)$ is constant and $\alpha(v)=a v$ satisfy the Gauss and Codazzi equations if and only if
\begin{equation}\label{eqn:rhoeqn}
\rho '(u)^2-\rho (u) \rho ''(u)=\frac{1}{4} H^2 \rho (u)^4-b^2
   e^{4 a u} \ .
\end{equation}
In particular, by Bonnet's theorem,  these data determine an intrinsic surface of revolution, and every such surface arises this way.
\end{corollary}
\begin{proof}
This follows by using the explicit solutions for $\lambda_1$ and $\lambda_2$ from Lemma \ref{lem:eigen} in the Gauss equation
from Lemma \ref{lem:gauss}, and simplifying.
\end{proof}

To classify all intrinsic surfaces of revolution, we would need to find all solutions to the differential equation \ref{eqn:rhoeqn}, and then to integrate the surface equation to obtain a parametrization. We will discuss the solutions of \ref{eqn:rhoeqn} for $H=0$ in Section \ref{sec:min}.

We end this section by carrying out the first integration step of the surface equation, which is quite explicit and shows that special coordinate curves are planar.

Assume that $\rho(u)$ is a solution of  \ref{eqn:rhoeqn}. To determine the surface parametrization, we will first determine a differential equation for the curve $\tilde c = f\circ c$ with $c(s) = (s,0)$.

Recall from Equations (\ref{eqn:UV}) and (\ref{lemma:levi}) that
\[
X(s) = U(s,0) \qquad\text{and}\qquad Y(s) = V(s,0)
\]
are a parallel frame field along $c(s)$ with respect to the first fundamental form. 

Following the proof of Bonnet's theorem, we derive a Frenet-type differential equation for the  orthonormal frame $\tilde X(s)= df X(s)$, $\tilde Y(s)= df Y(s)$, and $\tilde N(s) =\tilde X(s)\times \tilde Y(s)$. 

\begin{align*}
\tilde X'(s) {}&= df \Dds X(s) + \skp{\tilde X'(s)}{\tilde N(s)}\tilde N(s) \\
&= - \skp{\tilde X(s)}{\tilde N'(s)}\tilde N(s) \\
&= - \skp{df X(s)}{df S \frac{\partial}{\partial u}}\tilde N(s) \\
&= - \rho(s) I ( X(s), S X(s))\tilde N(s) 
\end{align*}

and similarly
\[
\tilde Y'(s) =   - \rho(s) I ( Y(s), S X(s))\tilde N(s)
\]

Finally,
\begin{align*}
\tilde N'(s) {}&=  \skp{\tilde N'(s)}{\tilde X(s)}\tilde X(s) + \skp{\tilde N'(s)}{\tilde Y(s)}\tilde Y(s)\\
&=  \skp{df S \frac{\partial}{\partial u}}{df X(s)}\tilde X(s) + \skp{df S \frac{\partial}{\partial u}}{df Y(s)}\tilde Y(s)\\
&=  I(S \frac{\partial}{\partial u}, X(s))\tilde X(s) + I(df S \frac{\partial}{\partial u},Y(s))\tilde Y(s)
\end{align*}

In our case, using the explicit formula for the shape operator and the principal curvatures in terms of $\rho$ and $a, b$, this simplifies to give the following lemma:

\begin{lemma}\label{lem:XYN}
\begin{align*}
\tilde X'(s) {}&= -\left(\frac{ e^{2 a s} b}{\rho (s)}+\frac12 H \rho (s)\right)\tilde N(s) \\
Y'(s) {}&=0 \\
\tilde N'(s) {}&= \left(\frac{ e^{2 a s} b}{\rho (s)}+\frac12 H \rho (s)\right)\tilde X(s) 
\end{align*}
\end{lemma}

\begin{corollary}
The space curve $f(s,0)$ is planar.
\end{corollary}
\begin{proof}
This is immediate because $\tilde Y(s)$ is constant. Note that this only works because $v=0$.
\end{proof}

This is as far as we can get in the general case. For the minimal case, we will solve the Equation (\ref{eqn:rhoeqn}) explicitly and be able to integrate the surface equations further.

\section{The minimal case}\label{sec:min}

In the minimal case $H=0$ the differential equation for $\rho$ simplifies to

\begin{equation}
\rho '(u)^2-\rho (u) \rho ''(u)=-b^2 e^{4 a u}
\end{equation}

Without much loss of generality, we can assume $b=1$ by  scaling $\rho$ by a positive constant.
There is one exception, namely when $b=0$. In this case, $\lambda_1=\lambda_2=0$, so that the surface is a plane, which we disregard.

\begin{lemma}\label{lem:rhosol}
All positive solutions of 
\begin{equation}\label{eqn:minimal}
\rho '(u)^2-\rho (u) \rho ''(u)=- e^{4 a u}
\end{equation} defined in any open interval are given by
 \[
 \rho(u) = \frac{e^{2 a u}}{2 B} \left(A e^{B u} +\frac{e^{-B
   u}}{A}\right)
 \]
 for arbitrary $A,B>0$.
\end{lemma}
\begin{proof}
It is easy to check that $\rho$ satisfies Equation (\ref{eqn:minimal}).
To show that every local solution $\sigma$ is of this form, it suffices to show that for any fixed real $u$, the initial values $\sigma(u)>0$ and $\sigma'(u)$ are equal to the initial data $\rho(u)$ and $\rho'(u)$ for a suitable choice of $A>0$ and $B>0$. Then the local uniqueness theorem for ordinary differential equations implies that $\rho =\sigma$ near $u$ and hence everywhere.

To this end, we have to solve

\begin{align*}
\sigma (u) {}&=\frac1{{2 B}}{e^{2 a u} \left(e^{B u} A+\frac{e^{-Bu}}{A}\right)} {} \\
\sigma' (u) {}&= \frac1{{2 B}} {e^{2 a u} \left(A B e^{B u}-\frac{B e^{-B u}}{A}\right)}
+\frac1{B} {a e^{2 au} \left(e^{B u} A+\frac{e^{-B u}}{A}\right)}
   \end{align*}
for $A$ and $B$. Surprisingly, this is explicitly possible. 

The strategy is to solve the first equation for $A$, choosing the larger solution of the two.
This gives
\[
A=e^{-B u} \left(B e^{-2 a u} \sigma(u)+\sqrt{B^2 e^{-4 a
   u} \sigma(u)^2-1}\right) \ .
   \]
Inserting this into the second equation and simplifying gives
\[
\sigma'(u)- 2 a \sigma(u) = \sqrt{B^2  \sigma(u)^2- e^{4 a u}}
\]
which can be solved for $B$. Again choosing the positive solution gives
\[
B=\frac1{\sigma(u)}  \sqrt{e^{4 a u}+ (\sigma'(u) - 2 a \sigma(u))^2} \ .
\]
Note that $\sigma(u)>0$ as we are only interested in positive conformal factors. This in turn makes the radicand in the preliminary expression for $A$, and hence $A$ itself,  positive. Explicity: 
\[
A= e^{-\frac{u \left(2 a \sigma (u)+\sqrt{\left(\sigma '(u)-2 a \sigma
   (u)\right)^2+e^{4 a u}}\right)}{\sigma (u)}} \left(-2 a \sigma
   (u)+\sigma '(u)+\sqrt{\left(\sigma '(u)-2 a \sigma
   (u)\right)^2+e^{4 a u}}\right)\ .
  \]
\end{proof}

\begin{remark}
The Enneper solution $\rho_{Enn}$ corresponds to $a=1$, $A=B=1$.
\end{remark}

Using the solutions for $\rho$ from Lemma \ref{lem:rhosol} in Lemma \ref{lem:XYN} (and remembering that we normalized $b=1$), straightforward computations give

\begin{align*}
\tilde X'(s) {}&= -\frac{2 A B e^{B s} }{A^2e^{2 B s} +1}\tilde N(s) \\
\tilde Y'(s) {}&=0 \\
\tilde N'(s) {}&= \frac{2 A B e^{B s} }{A^2e^{2 B s} +1}\tilde X(s) 
\end{align*}

Integrating gives the following lemma:
\begin{lemma}
Up to a motion in space, the solution to this equation is given by
\[
\tilde X(s) = \frac1{1+e^{2 B s} A^2}\begin{pmatrix}1-A^2 e^{2 B s}\\ 0  \\ -2 A e^{B s}\end{pmatrix},\ 
\tilde Y(s) = \begin{pmatrix}0\\ 1  \\0\end{pmatrix},  \ 
\tilde N(s) = \frac1{1+e^{2 B s} A^2}\begin{pmatrix}2 A e^{B s}\\ 0  \\1-A^2 e^{2 B s}\end{pmatrix}
\]
\end{lemma}

We have normalized the frame to that for $s=-\infty$, $\tilde X = (1,0,0)$ and $\tilde N = (0,0,1)$ in agreement with our parametrization of the Enneper surface.

\begin{corollary}
The space curve $\tilde c(s)=f(s,0)$ is given by
\[
\tilde c(s)=-\frac{e^{2 a s}}{2 B} \left(\frac{e^{B s} A}{B+2 a}+\frac{e^{-B s}}{A(B-2 a)},0,\frac{1}{a}\right)
\]
if $B \ne \pm2a$. If $B=2a$ (say, the other case being similar), we have
\[
\tilde c(s)=-\frac{e^{2 a s}}{4 a^2}\left(
\frac14 A e^{2 a s},0,1\right) +\frac{s}{4aA} (1,0,0)\ .
\]
\end{corollary}
\begin{proof}
This follows by integrating 
\begin{align*}
\tilde c'(s) & =\frac{d}{ds} f(s,0) \\
{}&= I(\frac{\partial}{\partial s}, X(s)) \tilde X(s) + I(\frac{\partial}{\partial s}, Y(s)) \tilde Y(s) \\
&=-\frac{e^{2 a s}}{2 B} \left(A e^{B s}-\frac{e^{-B s}}{A},0,2\right) 
\end{align*}
using the previous lemma, and simplifying.
\end{proof}

Instead of now integrating the surface equations likewise along the curves $s\mapsto (u,s)$ for fixed $u$, we will use the 
Bj\"orling formula \cite{dhkw} to obtain the parametrization more easily.

Recall that given a real analytic curve $\tilde c:(u_1,u_2)\to \R^3$ and a real analytic unit normal field $\tilde N:(u_1,u_2)\to \R^3$ satisfying $\skp{\tilde c'(u)}{\tilde N(u)}=0$, the unique minimal surface containing $\tilde c$ and having surface normal $\tilde N$ along $\tilde c$ can be given in a neighborhood of $(u_1,u_2) \subset \C$ by
\[
f(z) = \re \left( \tilde c(z) - i \int^z \tilde N(w)\times \tilde c'(w)\, dw \right)
\]
where we write $z=u+ i v$ and have extended $\tilde c$ and $\tilde N$ to holomorphic maps into $\C^3$.

In our case, we obtain  for $B\ne 2a$

\[
f(u,v) = \frac{e^{2 a u}}{2 B}
\begin{pmatrix}
\frac{e^{-B u} \cos ((2 a-B) v)}{2 a A-A
   B}-\frac{A e^{B u} \cos ((2 a+B) v)}{2
   a+B} \\
 \frac{e^{-B u} \sin ((2 a-B) v)}{2 a A-A
   B}+\frac{A e^{B u} \sin ((2 a+B) v)}{2
   a+B} \\
 -\frac{\cos (2 a v)}{a}
\end{pmatrix}
\]
and   for $B=2a$
\[
f(u,v) = \frac{1}{4a^2}
\begin{pmatrix}
\frac{a u}{A}-\frac{1}{4} A e^{4 a u}
   \cos (4 a v) \\
 \frac{a v}{A}+\frac{1}{4} A e^{4 a u}
   \sin (4 a v) \\
 -e^{2 a u} \cos (2 a v)  \end{pmatrix} \ .
\]

Note that in the last case scaling $a$ by  a constant and $(u,v)$ by the reciprocal only scales the surface, so we can as well assume that $a=1$ in this case.

The Weierstrass data \cite{dhkw} of these surfaces are particularly simple. Using  $z = u+i v$, 
let (also for $B=2a$)  
\[
G(z) =\frac{1}{A} e^{-B z}\qquad\text{and}\qquad dh  = -\frac{1}{B} e^{2 a z} \, dz \ .
\]

be the Gauss map and height differential of the Weierstrass representation formula

\[
f(z) = \re \int^z \begin{pmatrix} \frac12(1/G-G) \\ \frac{i}2 (1/G+G) \\ 1 \end{pmatrix}\, dh \ .
\]

This gives the surfaces $f(u,v)$ above. This can be verified either by evaluating the integral or by solving the Bj\"orling integrand
$\tilde c'(z) - i  \tilde N(z)\times \tilde c'(z) $ for $G$ and $dh$.

Of particular interest are the cases when $B$ and $2a$ are integers.
Then the substitution $z = -\log (w)$ changes the Weierstrass data into 

\[
G(z) =\frac{1}{A} w^{B}\qquad\text{and}\qquad dh  = \frac{1}{B} w^{-2 a-1} \, dw \ ,
\]
defined on the punctured plane $\C^*$ and being minimal surfaces of finite total curvature.

A substitution in the domain of the form $w \mapsto \lambda w$ will scale $G$ and $dh$ by powers of $\lambda$, so we can assume without loss of generality that $A=1$.

Some of the minimal surfaces we have obtained are described in \cite{ka5}.
We will now discuss examples.

\begin{figure}[H] 
   \centering
   \includegraphics[width=3in]{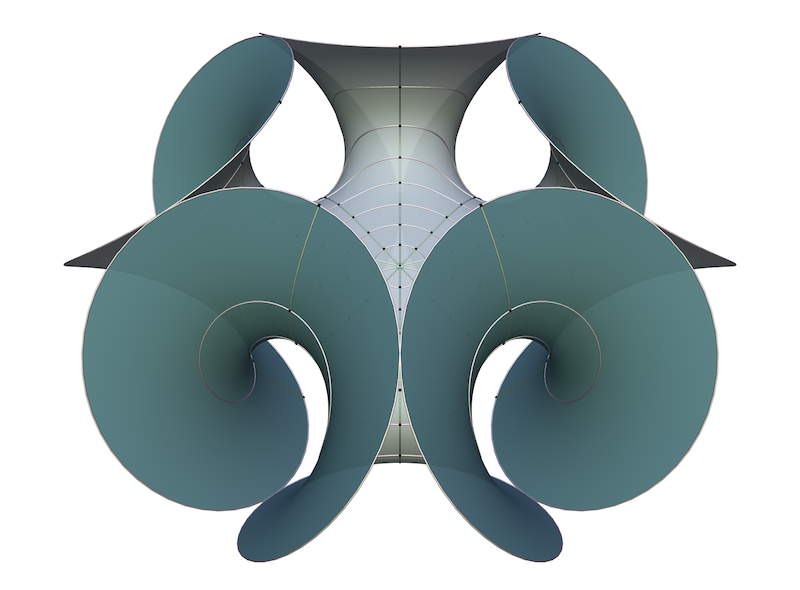} 
   \caption{The Enneper Surface of order 5}
   \label{fig:enneper5}
\end{figure}

In case that $B=2a-1\in\N$, we obtain the Enneper surfaces of cyclic symmetry of order $B+1$, see Figure \ref{fig:enneper5}.
For $B=2a-1=1$, we obtain the original Enneper surface.

The {\em planar Enneper surfaces} of order $n$ are given by choosing $B=n+1$ and $2a=n$.
See Figure \ref{fig:PlanarEnneper} for the cases $n=1$ and $n=6$. These surfaces feature an Enneper type end and a planar end.
Remarkably, in the non-zero CMC case, there are only one-ended solutions \cite{smy1}.

\def\fw{2.5in}
\begin{figure}[H]
 \begin{center}
   \subfigure[order 1]{\includegraphics[width=\fw]{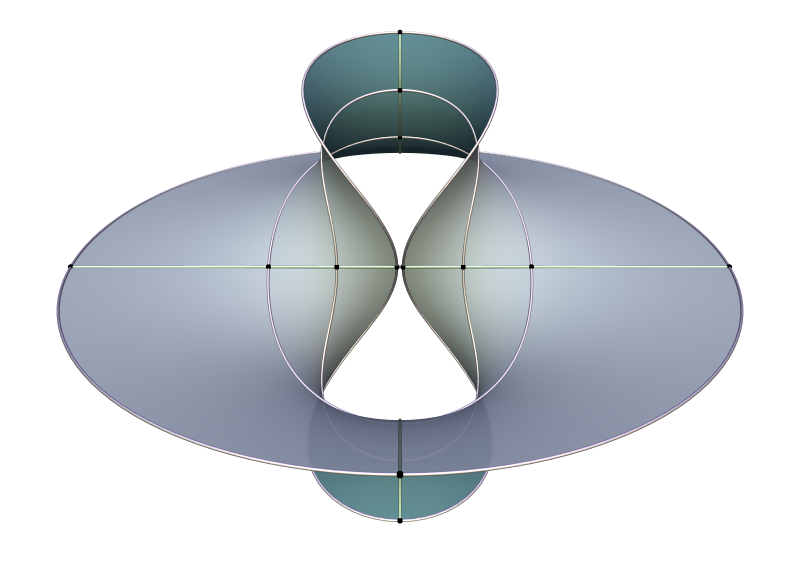}}
   \subfigure[order 6]{\includegraphics[width=\fw]{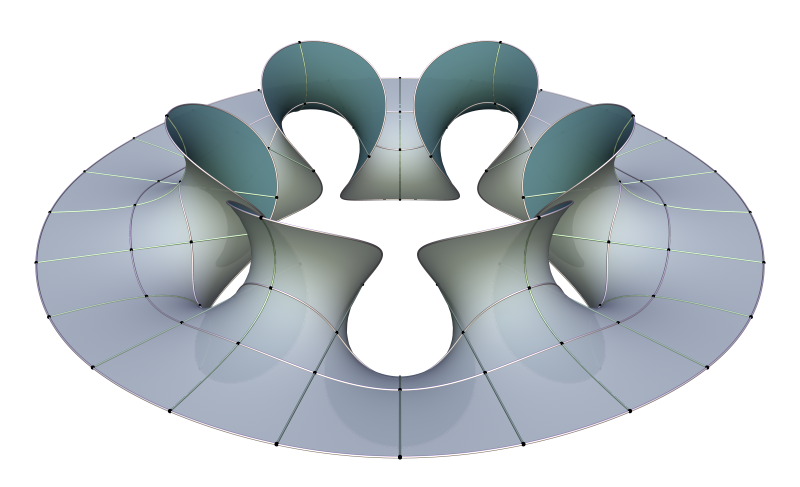}}
 \end{center}
 \caption{Planar Enneper surfaces}
 \label{fig:PlanarEnneper}
\end{figure}

Other choices of $a$ and $B$ lead to more wildly immersed examples. In Figure \ref{fig:GeneralEnneper} we show images of thin annuli $u_1<u<u_1$.

\def\fw{2.5in}
\begin{figure}[H]
 \begin{center}
   \subfigure[$B=1$, $a=3/2$]{\includegraphics[width=\fw]{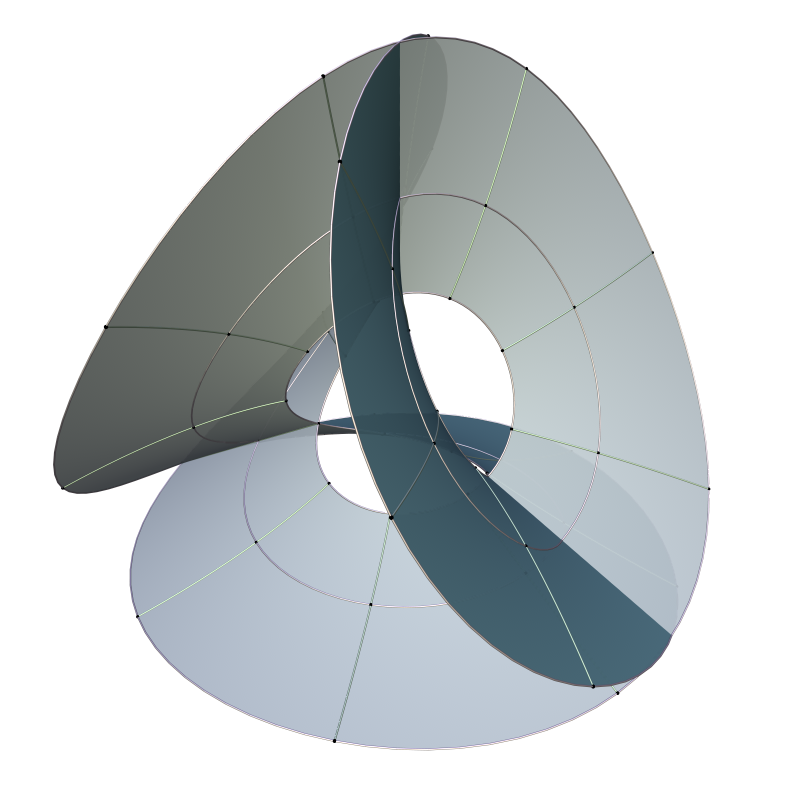}}
   \subfigure[$B=7$, $a=2$]{\includegraphics[width=\fw]{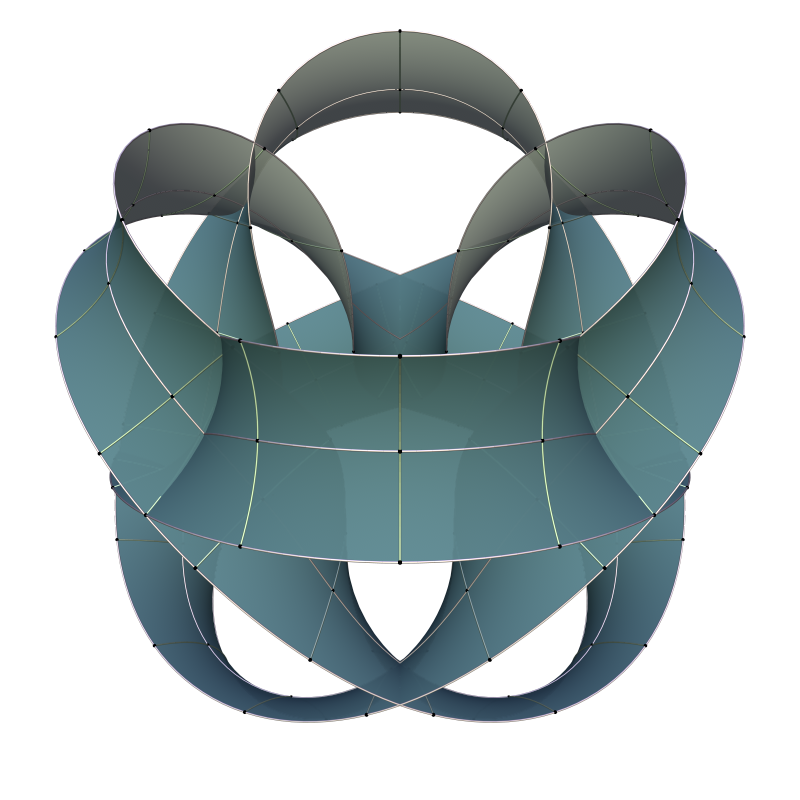}}
 \end{center}
 \caption{Generalized Enneper surfaces}
 \label{fig:GeneralEnneper}
\end{figure}

There is one  case that deserves attention: If $B=2a$, the Weierstrass 1-forms  have residues, and hence the surface can become periodic.
The prototype case here is $B=1$ and $a=1/2$ (see Figure \ref{fig:transenneper}) which leads to a translation invariant surface  that hasn't made it into the literature to our knowledge.
It deserves attention because it is in the potentially classifiable list of minimal surfaces in the space form $\R^3/\Z$ (where $\Z$ acts through a cyclic group of translations)
of finite total curvature $-4\pi$. Other surfaces in this list include the helicoid and the singly periodic Scherk surfaces.

\begin{figure}[H] 
   \centering
   \includegraphics[width=5in]{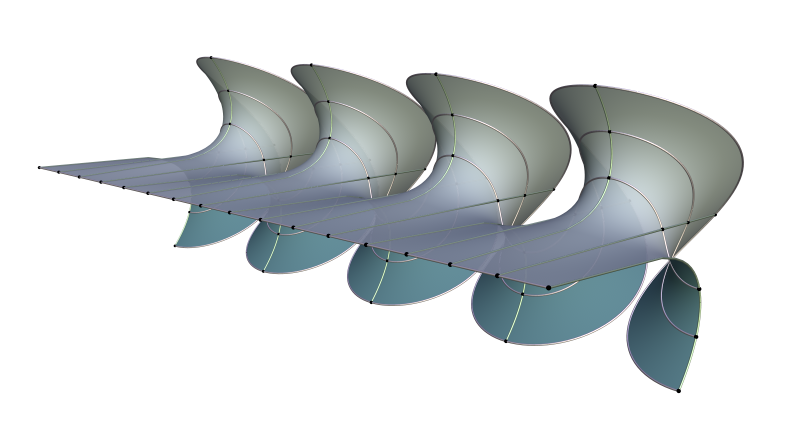} 
   \caption{The Translation Invariant Enneper Surface}
   \label{fig:transenneper}
\end{figure}

\section{Constant Mean Curvature}\label{sec:cmc}

\author{intrinscial?}
In \cite{smy1}, Smyth considers intrinsical surfaces of revolution under a different viewpoint: He assumes from the beginning that his surfaces have constant mean curvature, but does not make further assumptions about the shape operator. Nevertheless, we both end up with the same class of surfaces. Therefore we would like to connect our approach with Smyth's in the CMC case.

First we can compute the Hopf differential using the coordinate $z=u+i v$: Using the Definitions \ref{def:rho} of $I$ and \ref{def:shape} of $S$, and the formulas for $\alpha$, $\lambda_1$ and $\lambda_2$ from Theorem \ref{thm:ode}, a straightforward computation shows that

\begin{align*}
\Omega &= I\left(S\cdot \frac{d}{dz}, \frac{d}{dz}\right)\\
&= I\left(S\cdot \frac12\begin{pmatrix}1\\-i\end{pmatrix}, \frac12\begin{pmatrix}1\\-i\end{pmatrix}\right) \\
&= \frac12 b e^{2a u}\left( \cos(2av)+ i \sin(2a v)\right)\\
&= \frac12 b e^{2az}
\end{align*}

is indeed holomorphic and agrees with Smyth's computation. Secondly, to show that our equation for $\rho$ is equivalent to Smyth's equation, we substitute
\begin{align*}
\rho(u) &= e^{\phi(u)/2} \\
\phi(u) &= F(u) -2au+\log(b)
\end{align*}
and obtain
\[
F''(u) = -4 b e^{-2 a u} \sinh (F(u))
\]
in the case that $H=2$ (which is Smyth's case $H=1$). This again agrees with Smyth's equation, up to a normalization of constants.

In general, there are apparently no explicit solutions to Equation (\ref{eqn:rhoeqn}) for $H\ne 0$ in the literature. There is, however, one explicit solution given by 
\[
\rho(u) = \frac{\sqrt{2} \sqrt{b} e^{a u}}{\sqrt{H}} \ .
\]
By Lemma \ref{lem:eigen}, the principal curvatures become simply $\lambda_1 = H$ and $\lambda_2=0$. This implies that the surface under consideration is in fact a cylinder. This is somewhat surprising, as the standard parametrization of a cylinder over a circle of radius $1/H$ as an extrinsic  surface of revolution has twist 0. In our case, however, the cylinder is parametrized using geodesic polar coordinates (see the left image in Figure \ref{fig:cmc}) as 

\[
f(u,v) = \frac1H
\begin{pmatrix}
 \cos \left(\frac1a {\sqrt{2bH}  e^{a u} \cos (a v)}\right) \\
\sin \left(\frac1a {\sqrt{2b H}  e^{a u}  \cos(a v)}\right) \\
 \frac1a {\sqrt{2b H}  e^{a u}  \sin(a v)}\\
 \end{pmatrix}
\]

For other initial data of the Equation (\ref{eqn:rhoeqn}), only numerical solutions are available. These can be obtained easily by integrating the surface equations.
The  right image in Figure \ref{fig:cmc} was obtained using $a=1$, $b=4.2625$, $H=1/2$, and $\rho(0) =\rho'(0)=2$.

\def\fw{2.5in}
\begin{figure}[H]
 \begin{center}
   \subfigure[Cylinder in polar coordinates]{\includegraphics[width=\fw]{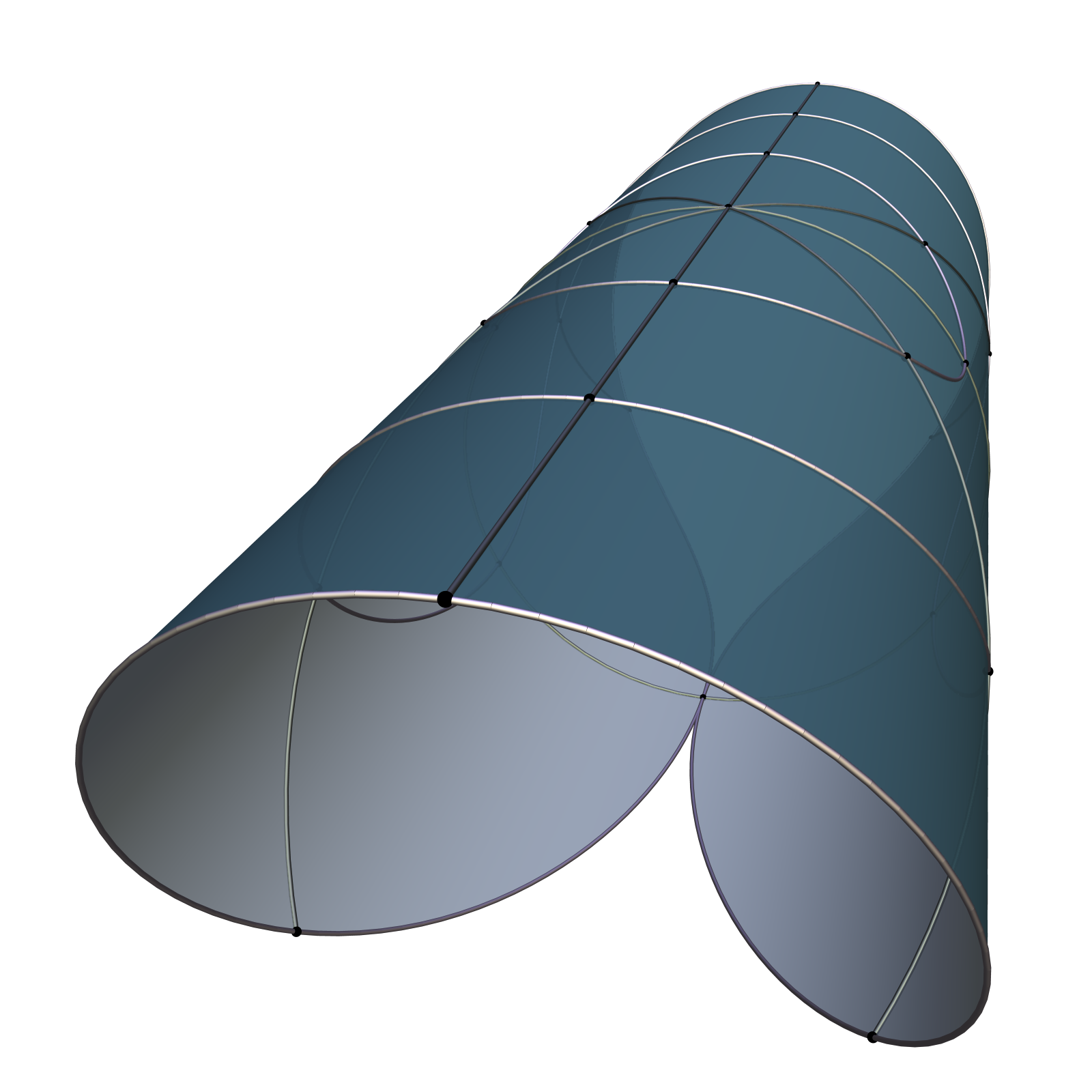}}
   \subfigure[Intrinsic CMC surface of revolution (numerical solution)]{\includegraphics[width=\fw]{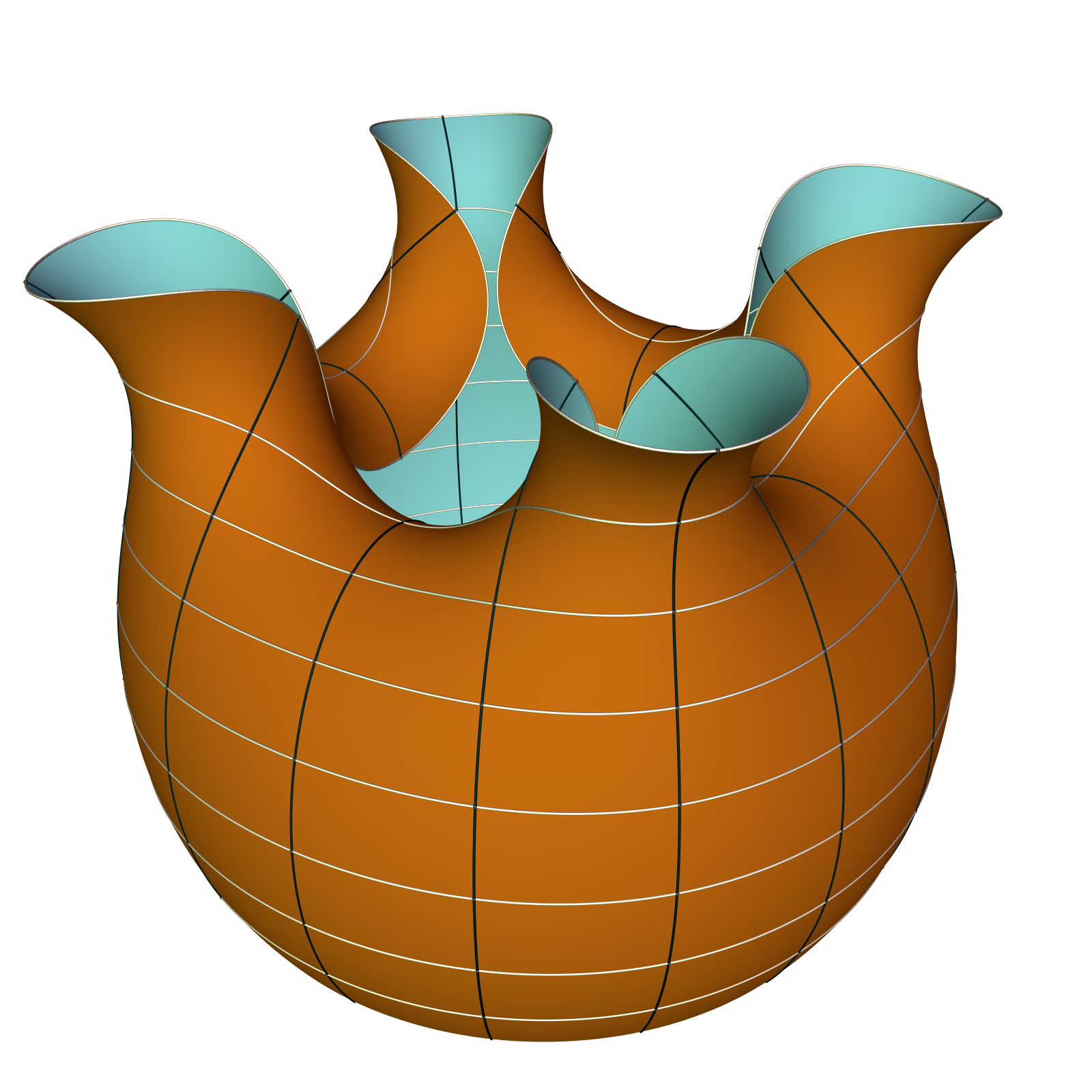}}
 \end{center}
 \caption{Two CMC surfaces}
 \label{fig:cmc}
\end{figure}

\section{The untwisted case}\label{sec:unt}

In this section, we will consider the exceptional case of Theorem \ref{thm:cmc} where $\alpha(v)=a$ with $a$ being an integral multiple of $\pi/2$, and prove Theorem \ref{thm:unt}.

Thus we are given a first fundamental form $I_\rho$ and shape operator

\begin{equation*}
S = \pm\begin{pmatrix}
\lambda_1(u) & 0 \\
 0 & \lambda_2(u) 
\end{pmatrix}  \qquad \text{or} \qquad
S = \pm\begin{pmatrix}
\lambda_2(u) & 0 \\
 0 & \lambda_1(u) 
\end{pmatrix}\ ,
\end{equation*}
depending on the congruence class of $a$ modulo $2\pi$. Without loss of generality, we will assume $a=0$ and therefore
\begin{equation*}
S = \begin{pmatrix}
\lambda_1(u) & 0 \\
 0 & \lambda_2(u) 
\end{pmatrix} \ .
\end{equation*}

The Gauss- and Codazzi equations become
\[
\lambda_1(u)\lambda_2(u) = \frac{\rho'(u)^2-\rho(u)\rho''(u)}{\rho(u)^4}
\]
and
\[
\frac{\rho'}\rho (\lambda_1-\lambda_2) {}=  \lambda_2'  \ .
\]

Eliminating $\lambda_1$ from the first equation using the second equation leads to the differential equation 
\[
\frac{\rho '(u)^2-\rho (u)
   \rho ''(u)}{\rho(u)^4}=\lambda_2(u) \left(\lambda_2(u)+\frac{\rho (u)
   \lambda_2'(u)}{\rho '(u)}\right)
 \]
 for $\lambda_2$. Surprisingly, this equation can be solved explicitly by

\begin{align*}
\lambda_1(u) {}&=\frac{\rho (u) \rho''(u)-\rho'(u)^2}{\rho (u)^2 \sqrt{c^2
   \rho (u)^2-\rho '(u)^2}}\\
   \lambda_2(u)  {}&=-\frac{\sqrt{c^2 \rho (u)^2-\rho'(u)^2}}{\rho (u)^2}
   \end{align*}
   for any choice of $c$ that makes the radicand positive.
   
 We now show that any untwisted surface is a general surface of revolution.  Recall that  typically a surface of revolution is being parametrized as
 \[
 f(u,v)=(g(u) \cos (v),g(u) \sin (v),h(u)) \ .
 \]
 
 However, by changing the speed of rotation, a surface of revolution can also be given by
  \[
 f(u,v)=(g(u) \cos (cv),g(u) \sin (cv),h(u))
 \]
 
 where $c$ is a positive constant.

 We now show that we can find $g$ and $h$ defined on the interval $(u_1, u_2)$ having the first fundamental form and shape operator of the untwisted intrinsic surface of revolution above, with the rotational speed-up $c$ being the constant $c$ in Theorem \ref{thm:unt} introduced above as an integration constant.
 
 The first fundamental form  of $f$ is given by:

  \[ 
I = \begin{pmatrix}
g'(u)^2+h'(u)^2 & 0 \\
 0 & c^2g(u)^2 
\end{pmatrix} \ .
 \]
 
 Comparing this to the definition of $I_\rho$ gives the following equations:
 
 \begin{align*}
 g'(u)^2+h'(u)^2 &= \rho(u)^2\\
  c^2 g(u)^2 &= \rho(u)^2 \ .
  \end{align*}
 
This determines  $g(u) = \frac{\rho(u)}{c}$ and $h(u)$ by  $h'(u)=\frac1c \sqrt{c^2\rho (u)^2-\rho'(u)^2}$. Note that the radicand is positive by our assumption about $c$.

Straightforward computation shows that the shape operator of $f(u,v)$ with $g$ and $h$ as above coincides with the shape operator $S$ of the intrinsic surface of revolution.

This completes the proof of Theorem \ref{thm:unt}.

 \begin{example} 
Knowing this, we can find  surfaces of revolution with speed-up $c\ge 3$ that are locally  isometric to the Enneper surface.    

For the Enneper surface, we have

\[
\rho(u) = \frac{1}{4} e^{2 u} \left(e^{-u}+e^u\right)
\] 
so that the radicand $c^2\rho (u)^2-\rho'(u)^2$ becomes
\[
\frac{1}{16} e^{2u}u \left(
\left(c^2-9\right) e^{4 u}
+\left(2 c^2-6\right) e^{2u}+
c^2-1
  \right) \ .
\]
Thus for $c\ge 3$, we can find $g$ and $h$ as needed. The integral for $h$ is generally not explicit, but for $c=3$ we can obtain

\begin{align*}
g(u) &= \frac{1}{12} e^{2 u} \left(e^{-u}+e^u\right) \\
h(u) &= \frac{1}{36} \left(2 \sqrt{3} \sinh
   ^{-1}\left(\sqrt{\frac{3}{2}} e^u\right)+3 e^u
   \sqrt{2+3 e^{2 u}}\right) \ . 
\end{align*}
  
This means that the surface of revolution in Figure \ref{fig:ennperrevolve} is isometric to one third of the Enneper surface, punctured at the ``center''.

\begin{figure}[H] 
   \centering
   \includegraphics[width=5in]{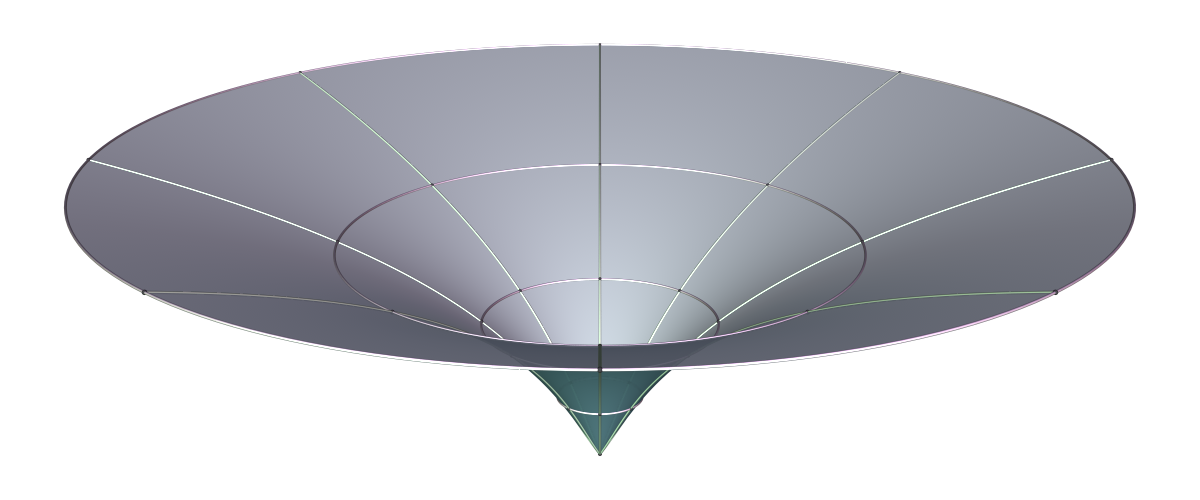} 
   \caption{Surface of revolution isometric to one third of the  Enneper Surface}
   \label{fig:enneperrevolve}
\end{figure}

In contrast, if $c=1$, the radicand is negative for all $u$, which implies that no piece of the Enneper surface can be isometrically realized as a standard surface of revolution (with no speed-up).
 \end{example}

\bibliographystyle{plain}
\bibliography{bibliography}

\end{document}